\documentclass{article}
\usepackage{graphicx} 

\title{On a Simple Continuation for Partial Sums}
\author{Kamal Saleh}
\date{January 2024}
\usepackage{tikz}
\usetikzlibrary{arrows,decorations.markings,patterns,calc}
\usepackage[heightadjust=all]{floatrow}
\usepackage[labelformat=empty]{caption}
\usepackage[utf8]{inputenc}
\usepackage[section]{placeins}
\usepackage{hyperref}
\usepackage{biblatex}
\usepackage{csquotes}
\usepackage{amsmath}
\usepackage{amssymb}
\usepackage{amsthm}
\usepackage[utf8]{inputenc}
\usepackage{graphicx}
\DeclareSymbolFont{bbold}{U}{bbold}{m}{n}
\DeclareSymbolFontAlphabet{\mathbbold}{bbold}
\usepackage[english]{babel}
\newtheorem{theorem}{Theorem}[section]
\newtheorem{lemma}[theorem]{Lemma}
\begin{document}

\maketitle

\begin{abstract}
    In 2014, Ibrahim M Alabdulmohsin wrote a paper called "Summability Calculus" where he developed a method to generalize sigma notation to non-integer upper bounds. His paper included a theorem, known as Theorem 6.1.1 (denoted here as Lemma 2.1 because of its simplicity and location in this paper), but doesn't study it much. Another paper by Mueller and Schleicher also analyzed this formula, but doesn't integrate or differentiate the formula and states some specific applications. This paper will analyze the simple formula that generalizes sigma notation to non-integer upper and lower bounds. We state and prove this formula in Section 2. Section 3 states a few algebraic properties for the sum and product formulae and shows how differentiation of sums and products works. Because integrating a product is challenging, we only analyze the integration of sums in the fourth part of Section 3. In Section 4 we apply the formula in the second section to create analytic continuations for functions defined as partial sums, formulate an infinite series representation to any limit, create a great approximation for functions that approach a certain limit, make an analytic continuation for products, and calculate the sum of anti-derivatives. We then conclude with a discussion of the material of this paper.
\end{abstract}

\section{Introduction}
One of the most important concepts in Mathematics is generalization. Whenever a new concept is introduced, mathematicians instantly try to generalize it. That way a theorem, rule, etc would have much more use. 

A method of generalization includes continuation, which is the expansion of a function's domain (i.e. generalizing it to more numbers). This was first studied by Bernhardt Riemann when he analyzed the Riemann Zeta function, which will be used in this paper. Euler (or Maclaurin) also made a formula that generalizes the computation of sums for non-integer upper and lower bounds (the original purpose of this paper), which is called the Euler-Maclaurin forumla. Alabdulmohsin wrote the paper "Summability Calculus" \cite{alabdulmohsin2012summability}, which states how Discrete and Continuous Calculus could be united under one field, devotes many of its sections to generalizing summation, specifically its second chapter.

However, many of the known explicit formulae for partial sums are very complicated, usually involving complicated sequences like the Bernoulli coefficients given in the Euler-Maclaurin formula. The primary focus of this paper is to study partial sums with a simple formula (which we'll denote as Lemma 2.1) that evaluates the partial sums of sums with non-integer upper and lower bounds. Note that Lemma 2.1 has been studied already by Alabdulmohsin\cite{alabdulmohsin2012summability} and Mueller and Schleicher\cite{mueller2007fractional}, but there are some properties and applications not covered in these papers as well as treating the lower bound as a constant or parameter rather than an independent variable.

Before applying Lemma 2.1, we first study its algebraic properties, which are done to verify if they satisfy the rules of Summability Calculus. Then, we cover differentiation of Lemma 2.1 not only with respect to the upper bound but the lower one as well, which was mostly ignored by previous literature. Because we now treat sums as a function of two variables, Partial differential equations are created from the derivatives of sums. Also, using the resulting derivatives, we can express sums as Taylor series.

This paper also touches on continuating products and their properties and differentiation as well, but the primary focus of this paper is on sums. So products don't go past the PDEs section of the paper.

Lemma 2.1 will also be integrated, again with respect to both the upper and lower bounds. A few examples are stated for them as well. After covering integration we enter the Applications section, where the original purpose of Lemma 2.1 is utilized to make a few Analytic Continuations that haven't been covered by Alabdulmohsin and Mueller and Schleicher for a sum that is studied (the generalized harmonic numbers) and another one not as famous (the alternating harmonic numbers) but is very interesting.

After using Lemma 2.1 for its original purpose, we apply the Euler-Maclaurin formula to come up with a function approximation in terms of its derivative. After this we extend Lemma 2.1 by applying the PDEs covered in the Properties Section to derive a formula for summing antiderivatives with respect to the sum of the original function. We conclude this section with two unrelated but interesting formulas for partial sums, and then conclude the entire paper with a discussion.

\section{The Simple Result}
\subsection{Base Formula}
We now state a simple continuation for partial sums. It has been proven more generally in Summability Calculus\cite{alabdulmohsin2012summability}, but we give a proof without its mechanics for those that are unfamiliar.

\begin{lemma} If $\lim_{k\rightarrow\infty}f(k)=L$ for some constant $L$ and function $f(x)$ defined on the positive integers, then 
$$\sum_{k=1}^xf(k)=xL+\sum_{k=1}^\infty(f(k)-f(k+x))$$\end{lemma}

\begin{proof} We analyze the following partial sum: $$R(x,n)=\sum_{k=1}^n(f(k)-f(k+x))$$This could be rewritten as $$R(x,n)=\sum_{k=1}^nf(k)-\sum_{k=1}^nf(k+x)=\sum_{k=1}^nf(k)-\sum_{k=x+1}^{x+n}f(k)$$Let $n>x+1$ (the goal is to take the limit of $R$ as $n\rightarrow\infty$), then we can write that $$\sum_{k=1}^nf(k)-\sum_{k=x+1}^nf(k)-\sum_{k=n+1}^{n+x}f(k)=\sum_{k=1}^xf(k)-\sum_{k=n+1}^{n+x}f(k)$$As $n$ approaches infinity, the first sum remains the same, but the second one approaches $xL$, as there are $x$ functions in the sum, each approaching $L$. Therefore, we have that $$\sum_{k=1}^xf(k)=xL+\sum_{k=1}^\infty(f(k)-f(k+x))$$\end{proof}

The problem is, this only works when $x\in\mathbb{Z}$. We could prove the infinite series converges (and is monotonic) for any $x$ such that $h+x$ is in $f(x)$'s domain for some integer $h$ if $f(x)$ is monotonic by bounding it as follows: $$\sum_{k=0}^\infty(f(k)-f(\lfloor x\rfloor+k))\le\sum_{k=0}^\infty(f(k)-f(x+k))\le\sum_{k=0}^\infty(f(k)-f(\lceil x\rceil+k))$$
Assuming that it is monotonic decreasing. If $f(x)$ is monotonic increasing, switch the symbols.

There is a supposed contradiction. If we let $x=1$, then the infinite series would be $(f(0)-f(1))+(f(1)-f(2))+(f(2)-f(3))\pm...$. The terms seem to cancel if we write it as $f(0)+(-f(1)+f(1))+(-f(2)+f(2))+...$, getting that $f(0)=f(0)+L$ which implies that $L$ must be $0$. But these two sums don't need to be equivalent because of the Riemann Rearrangement theorem, which states that the associative and commutative properties don't need to hold for infinite series. Letting $x$ be any other integer will result in the same observation.

For Lemma 2.1 to work, $f(x)$ has to converge to a constant and its domain consists of all positive integers. If $x$ is not an integer, then it is sufficient for $f(x)$ to be monotonic. It is also sufficient if $f(x)=\cos(\pi x)g(x)$ where $g(x)$ is a monotonic function. This is because we can rewrite the finite sum as $$\sum_{k=1}^xf(k)=\sum_{k=0}^{\left\lfloor\frac x2\right\rfloor}g(2k)-\sum_{k=1}^{\left\lfloor\frac {x+1}2\right\rfloor}g(2k-1)$$ These aren't necessary conditions though, as a function like $\frac{\sin x}x$ which isn't monotonic and can't be written as $\cos(\pi x)$ times a monotonic function still has a convergent series.
\subsection{General Formula}
Lemma 2.1 can be generalized. Subtracting $\sum_{k=1}^{y-1}$ from $\sum_{k=1}^x$ we get $$\sum_{k=y}^xf(k)=L(x-y+1)+\sum_{k=1}^\infty(f(k+y-1)-f(k+x)$$

\subsection{Product Formula}

We could apply Lemma 2.1 to get a formula for products as well. So, as always, let $f(x)$ be a monotonic function that is greater than $0$. Then for all real numbers $x$ in the domain of $f$, we have the following: $$\prod_{k=1}^xf(k)=L^x\prod_{k=1}^\infty\frac{f(k)}{f(k+x)}$$.

Which can be proved like so:$$\prod_{k=1}^xf(k)=e^{\sum_{k=1}^x\log(f(k))}=e^{x\log L+\sum_{k=1}^\infty\log\left(\frac{f(k)}{f(k+x)}\right)}=L^x\prod_{k=1}^\infty\frac{f(k)}{f(k+x)}$$

The product doesn't work when $f(x)$ approaches a negative number because the infinite product won't exist or it becomes zero. The proof also implies that $L>0$, so this analytic continuation won't work when $L=0$. If $x=0$, the product equates to $1$.

Here is an example: Let $f(x)=\left(1+\frac1x\right)^x$, then, after solving for $e^x$: $$e^x=\frac{\prod_{k=1}^x\left(1+\frac1k\right)^k}{\prod_{k=1}^\infty\frac{\left(1+\frac1k\right)^k}{\left(1+\frac1{k+x}\right)^{k+x}}}$$When $x$ is an integer. As can be seen, solving for $L^x$ doesn't achieve much, but it can lead to interesting identities.

We can generalize the product as follows: $$\prod_{k=y}^xf(k)=L^{x-y+1}\prod_{k=1}^\infty\frac{f(k+y-1)}{f(k+x)}$$
\section{Properties}
\subsection{Fundamental Properties for the Sum and Product}
Following Lemma 2.1, we get two properties for the analytic continuation of sums and products, which are $\sum_{k=y}^{y-1}f(k)=0$ and $\prod_{k=y}^{y-1}f(k)=1$. Note here that the sum and product notations are no longer restricted to the integers. These properties were stated as the Empty-Sum and Empty-Product rules in the paper "Summability Calculus". The sum and product formulae follow two recurrence relations: $$\sum_{k=y}^xf(k)=f(y)+\sum_{k=y+1}^xf(k)=f(x)+\sum_{k=y}^{x-1}f(k)$$$$\prod_{k=y}^xf(k)=f(y)\prod_{k=y+1}^xf(k)=f(x)\prod_{k=y}^{x-1}f(k)$$Two more properties are defined below:$$\sum_{k=y}^xf(k)=\sum_{k=y}^cf(k)+\sum_{k=c+1}^xf(k)$$$$\prod_{k=y}^xf(k)=\prod_{k=y}^cf(k)\prod_{k=c+1}^xf(k)$$Where $c$ and $c+1$ don't necessarily need to be between $y$ and $x$. Creating these properties are straightforward after repeated application of one of the recurrence relations for the Sum/Product. These are similar to the first fundamental theorem of Calculus and is also fundamental to the properties of the many of the identities that were derived in Summability Calculus. The paper where this field of Math was created contains properties for finite sums\cite{alabdulmohsin2012summability} that also hold true by applying Lemma 2.1. A few have already been stated in the previous section, but there are two more that come from this book (re-written):$$\sum_{k=y}^xf(k)=-\sum_{k=x+1}^{y-1}f(k)$$$$\prod_{k=y}^xf(k)=\prod_{k=x+1}^{y-1}\frac1{f(k)}$$
The first one was explicitly stated in the book (all the properties that come from the book are stated in page 17) in a different form, which we could show by adding both sides by the RHS, applying the fundamental rule for sums, and then applying the Empty sum rule. The product rule can be derived similarly.
\subsection{Differentiation}
It would be interesting to find how differentiation and integration is defined upon these new generalized operators. 

\subsubsection{Upper Bound} We will consider differentiation with respect to the upper bound and then with respect to the lower. Differentiating both sides of Lemma 2.1, we get $$\frac{\partial}{\partial x}\sum_{k=y}^xf(k)=L-\sum_{k=1}^\infty f'(k+x)$$Where the infinite series can be proven to be convergent by the integral test.

Obtaining an expression for the $n$th derivative of this function is straightforward, so after doing so we could derive a Taylor series expansion for $\sum_{k=y}^xf(k)$ with respect to $x$ and around $y-1$. This results in $$\sum_{k=y}^xf(k)=L(x-y+1)-\sum_{k=1}^\infty\frac{(x-y+1)^k}{k!}\sum_{n=0}^\infty f^{(k)}(n+y)$$The reason we make it centered around $y-1$ is so the first term of the Taylor Series becomes zero by the empty sum rule. We could use this to easily prove that $$H_x=\sum_{k=2}^\infty(-1)^k\zeta(k)x^{k-1}$$which is a well known expansion for $H_x$ where $|x|<1$. Integrating this from $0$ to $1$, we get $$\int_0^1 H_xdx=\sum_{k=2}^\infty\frac{(-1)^k\zeta(k)}k$$ we will prove that this is equal to $\gamma$ in the section for Integration Sums. The generalized harmonic numbers satisfy $$H_x^{(m)}=\sum_{n=1}^\infty\frac{(-1)^{n+1}(m+n-1)_{m-1}\zeta(m+n)x^n}{n!}$$Integrating this expansion from $0$ to $1$: $$\int_0^1H_x^{(m)}dx=\sum_{n=2}^\infty\frac{(-1)^n(m+n-2)_{m-1}\zeta(m+n-1)}{n!}$$Which will also be in the Integrating Sums section. These two expansions are already known but can be derived quickly using the new derivative expression.

Computing the derivative for products is a bit more tricky. After repeatedly applying the product rule to one factor of the product at a time, we get that $$\frac{\partial}{\partial x}\prod_{k=y}^xf(k)=\left(\ln L-\sum_{k=1}^\infty\frac{f'(k+x)}{f(k+x)}\right)\prod_{k=y}^xf(k)$$Obtaining an $n$th derivative for this expression can be obtained by the Leibniz rule for differentiation, but it isn't feasible. Finding a Taylor Series for products is therefore difficult.

\subsubsection{Lower Bound}
Taking the derivative with respect to the lower bound is very similar. We have that $$\frac{\partial}{\partial y}\sum_{k=y}^xf(k)=-L+\sum_{k=0}^\infty f'(k+m)$$Note that we start at $k=0$. From this we obtain the Taylor series for $\sum_{k=y}^xf(k)$ with respect to $y$ and centered around $x+1$: $$\sum_{k=y}^xf(k)=L(y-x+1)+f(x)+\sum_{k=1}^\infty\frac{(y-x-1)^k}{k!}\sum_{n=0}^\infty f^{(k)}(n+x)$$ The derivative of the product with respect to the lower bound is $$\frac{\partial}{\partial y}\prod_{k=y}^xf(k)=\left(-\ln L+\sum_{k=0}^\infty\frac{f'(k+y)}{f(k+y)}\right)\prod_{k=y}^xf(k)$$
\subsubsection{PDEs}
Using the previous results, we can make four PDEs:$$\left(\frac{\partial}{\partial x}+\frac{\partial}{\partial y}\bigg|_{y=x+1}\right)\sum_{k=y}^xf(k)=\left(\frac{\partial}{\partial x}\bigg|_{x=y-1}+\frac{\partial}{\partial y}\right)\sum_{k=y}^xf(k)=0$$$$\left(\frac{\partial}{\partial x}+\frac{\partial}{\partial y}\bigg|_{y=x+1}\right)\prod_{k=y}^xf(k)=\left(\frac{\partial}{\partial x}\bigg|_{x=y-1}+\frac{\partial}{\partial y}\right)\prod_{k=y}^xf(k)=0$$Where $\left(\frac{\partial}{\partial x}+\frac{\partial}{\partial y}\bigg|_{y=a}\right)f(x,y)$ is interpreted as the partial derivative of $f(x,y)$ with respect to $x$ plus the partial derivative of $f(x,y)$ with respect to $y$, after which all the $y$'s are replaced with $a$. This operator isn't used much so it isn't clear whether these equations should be defined as PDEs. We can, however, make true PDEs using these results. Through Alabdulmohsin's result $$\frac{\partial}{\partial x}\sum_{k=y}^xf(k)=\sum_{k=y}^xf'(k)+\frac{\partial}{\partial x}\sum_{k=y}^xf(k)\bigg|_{x=y-1}$$Implying that $$\left(\frac{\partial}{\partial x}+\frac{\partial}{\partial y}\right)\sum_{k=y}^xf(k)=\sum_{k=y}^xf'(k)$$We also have that $$\frac{\partial^2}{\partial x\partial y}\sum_{k=y}^xf(k)=0$$This can be proven without the explicit definition of the sum (and thus for any $f$) by using the Fundamental Theorem of Summability Calculus and letting $c$ be a constant not dependent on $x$ or $y$.

We can also make two PDEs for Products, although they aren't as nice as that for sums: $$\left(\frac{\partial}{\partial x}+\frac{\partial}{\partial y}\right)\prod_{k=y}^xf(k)=\sum_{k=y}^x\frac{f'(k)}{f(k)}\prod_{k=y}^xf(k)$$We also have$$\prod_{k=y}^xf(k)\frac{\partial^2}{\partial x\partial y}\prod_{k=y}^xf(k)=\frac{\partial}{\partial x}\prod_{k=y}^xf(k)\frac{\partial}{\partial y}\prod_{k=y}^xf(k)$$
\subsection{Integrating Sums}
As a disclaimer, we will only be analyzing the integral of the sum, as calculating integrals for infinite products is not feasible. Integrating from a constant $a$ to $x$ with respect to the upper bound gives $$\int_a^x\sum_{k=y}^tf(k)dt$$$$=\frac{L((x-y+1)^2-(a-y+1)^2)}2+\sum_{k=1}^\infty\left(f(k+y-1)(x-a)+\int_{k+a}^{k+x}f(t)dt\right)$$Because $\sum_{k=y}^xf(k)$ is a continuous monotonic function, we are allowed to switch the integral and summation symbols. Note here that the infinite sum converges. 

The integral with respect to the lower bound is very similar: $$\int_a^y\sum_{k=t}^xf(k)dt$$$$=\frac{L((x-a+1)^2-(x-y+1)^2)}2-\sum_{k=0}^\infty\left(\int_{k+a}^{k+y}f(t)dt+f(k+x+1)(y-a)\right)$$

Two examples for integrating with respect to the upper bound are given below:$$\int_0^1H_xdx=\gamma$$$$\int_0^1H_x^{(m)}dx=\zeta(m)-m+1$$If we let $k=1$ in the second integral, we get that we could "assign" the value $\gamma$ to $\zeta(1)$, even though the zeta function has a pole here. This agrees with the Ramanujan summation of the harmonic series, which gives the value $\gamma$. It doesn't agree with the Ramanujan summation for $k=-1$ however, as the integrand no longer approaches zero and so the formula doesn't hold.

Using the integration of the Taylor series expansions of these analytic continuations in the Differentiation section, we get $$\sum_{k=2}^\infty\frac{(-1)^k\zeta(k)}k=\gamma$$$$\sum_{n=2}^\infty\frac{(-1)^n(k+n-2)_{k-1}\zeta(k+n-1)}{n!}=\zeta(k)-k+1$$It is important to note that the representation for $\gamma$ is already known (as well as the integral), while the representation for the zeta function hasn't been documented.
\section{Applications}
This section includes using Lemma 2.1 for analytic continuations for sums, creating approximations for functions, and summing integrals.

\subsection{Analytic Continuations}
The sum $\sum_{k=y}^xf(k)$ as traditionally defined exists only when $x,y\in\mathbb{Z}$ while $L(x-y+1)+\sum_{k=1}^\infty(f(k+y-1)-f(k+x))$ exists if the input is in the domain of $f(x)$. An example is $f(k)=\frac1k$, already found by Alabdulmohsin, getting the famous analytic continuation of the Harmonic numbers: $$H_x=\sum_{k=1}^\infty\frac x{k(k+x)}$$The Generalized Harmonic Numbers would be $$H_x^{(m)}=\zeta(m)-\zeta(m,x)+\frac1{x^m}$$ Where $\zeta(k,x)$ is the Hurwitz Zeta function defined as $\zeta(s,a)=\sum_{k=0}^\infty\frac1{(n+a)^s}$. Note here that $m>1$ is a must for this formula to work. These two analytic continuations were used to make the results proven in the Properties section.

Remember that $f(k)$ doesn't need to be monotonic for the infinite series in Lemma 2.1 to converge. For example, let $f(k)=\frac{(-1)^{k+1}}k$, then the Alternating Harmonic numbers would be $$\bar{H}_x=\sum_{k=1}^\infty\left(\frac{(-1)^{k+1}}k+\frac{(-1)^{k+x}}{k+x}\right)=\ln2+(-1)^x\sum_{k=1}^\infty\frac{(-1)^k}{k+x}$$ This infinite sum is equal to $\frac12\left(\psi\left(\frac{x+1}2\right)-\psi\left(\frac x2\right)\right)-\frac1x$, where $\psi$ is the Digamma function. This is proved by using the alternating symbol to get $$\sum_{k=0}^\infty\frac{(-1)^k}{k+x}=\frac12\sum_{k=0}^\infty\left(\frac1{k+\frac x2}-\frac1{k+\frac{1+x}2}\right)$$Then, we can add and subtract $\frac1{k+1}$ to the summand. Splitting the sums we get $\frac12\left(\psi\left(\frac{x+1}2\right)-\psi(\frac x2)\right)$. So: $$\bar{H}_x=\ln2+(-1)^x\left(\frac12\left(\psi\left(\frac{x+1}2\right)-\psi\left(\frac x2\right)\right)-\frac1x\right)$$to make this a true analytic continuation, we could replace $(-1)^x$ with $\cos(\pi x)$ because $(-1)^x=\cos(\pi x)$ when $x$ is an integer. Therefore: $$\bar{H}_x=\ln2+\cos(\pi x)\left(\frac12\left(\psi\left(\frac{x+1}2\right)-\psi\left(\frac x2\right)\right)-\frac1x\right)$$Before simplifying we had that $\bar{H}_0=0$ so we define the above formula as $\bar{H}_x$ where $x\ne0$ and $0$ if it is zero. When $x$ is a half integer, $\cos(\pi x)$ is zero, so $\bar{H}_x=\ln2$ whenever $x=k+\frac12$ for some integer $k$. This is equivalent to stating that the zeroes of $\bar{H}_x-\ln2$ are the half integers. Finding the zeroes of $\bar{H}_x$ is more difficult. By graphical observation, the following seems to hold:
\begin{theorem}
    Let $x_n$ denote the $n$th largest root of $\bar{H}_x$. Then $\lim_{n\rightarrow\infty}\{x_n\}=-\frac1{\pi}\arctan\left(\frac\pi{\ln2}\right)$, where $\{x\}$ is the fractional part of $x$.
\end{theorem}
\begin{proof}
    We first make a reflection formula for $\bar{H}$:$$\bar{H}_x-\bar{H}_{2-x}=\pi\cot(\pi x)-\frac{x^2-2x+2}{x(x^2-3x+2)}\cos(\pi x)$$Which can be proven by replacing $\frac12\left(\psi\left(\frac{x+1}2\right)-\psi\left(\frac x2\right)\right)$ with $\psi(x)-\psi\left(\frac x2\right)-\ln2$, an identity proven by differentiating the logarithm of Gauss's duplication formula for the Gamma function. Adding $\ln2$ on both sides and taking $x\rightarrow-\infty$, we get that $\bar{H}_x\sim\pi\cot(\pi x)+\ln2$. Letting $x=x_n$ we have $0\sim\pi\cot(\pi x_n)+\ln2$. By the Intermediate Value Theorem, we can prove that there exists one zero of $\bar{H}_x$ between the negative integers, so $x_n\sim-n-\frac1{\pi}\arctan\left(\frac\pi{\ln2}\right)$. Taking the fractional part gives the answer.
\end{proof}
\subsection{Function Approximation}
A similar formula to Lemma 2.1 is known as the Euler-Maclaurin Summation formula. A special case of it is $$\sum_{k=y}^xf(k)=\int_y^xf(t)+\left(t-\lfloor t\rfloor-\frac12\right)f'(t)dt+\frac{f(x)+f(y)}2$$ Setting it equal to Lemma 2.1: $$L(x-y+1)+\sum_{k=1}^\infty(f(k+y-1)-f(k+x))=\int_y^xf(t)+\left(t-\lfloor t\rfloor-\frac12\right)f'(t)dt+\frac{f(x)+f(y)}2$$ Taking the derivative with respect to $x$ on both sides and solving for $f(x)$: $$f(x)=L+(\lfloor x\rfloor-x)f'(x)-\sum_{k=1}^\infty f'(k+x)$$which can be re-written as $$f(x)=(\lfloor x\rfloor-x)f'(x)+\frac{\partial}{\partial x}\sum_{k=y}^xf(k)$$If we instead took the derivative with respect to $y$ then $$f(y)=(\lfloor y\rfloor-y)f'(y)-\frac{\partial}{\partial y}\sum_{k=y}^xf(k)$$

This formula doesn't hold for all $n$ as the integrand we worked with was discontinuous, but what is interesting is that the two sides of this equation share similar asymptotic growth on the positive real line and when $f$ is monotonic. In fact, the integral of both sides are even more similar, if not equal, when the bounds are positive integers. There are also some non-monotonic functions that satisfy these properties as well. Some images are shown on the next page.

The last two graphs show that the approximation can also be accurate for the negative real line, but the first one shows that this is not always the case. The last function $\frac{\sin x}x$ isn't monotonic but the approximation still works out fine. This approximation can also be useful for $\sin x$ by multiplying the approximation of $\frac{\sin x}x$ by $x$.
\subsection{Summing Integrals}
A limitation to the power of the generalized Lemma 2.1 is that $f(x)$ must approach a constant $L$ as $x$ approaches infinity. But, using the PDEs defined in the Properties section, we can extend Lemma 2.1 to the integral of $f(x)$, the integral of the integral of $f(x)$, and so on. To start, we note that $$\left(\frac{\partial}{\partial x}+\frac{\partial}{\partial y}\right)\sum_{k=y}^xF(k)=\sum_{k=y}^xf(k)$$ 
\FloatBarrier
\begin{figure}
    \centering
    \includegraphics[width=7cm]{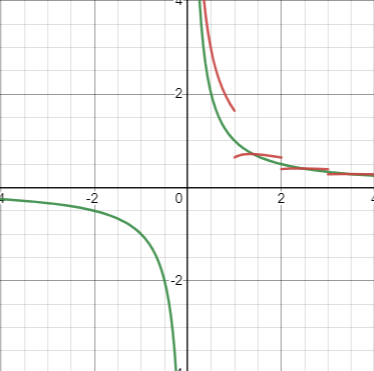}
    \caption{Approximation for $1/x$}
    \label{figure 1}
\end{figure}
\begin{figure}
    \centering
    \includegraphics[width=7cm]{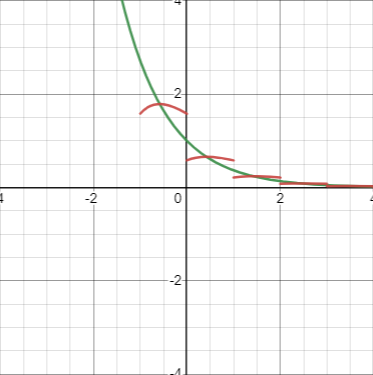}
    \caption{Approximation for $e^{-x}$}
    \label{figure 2}
\end{figure}
\begin{figure}
    \centering
    \includegraphics[width=7cm]{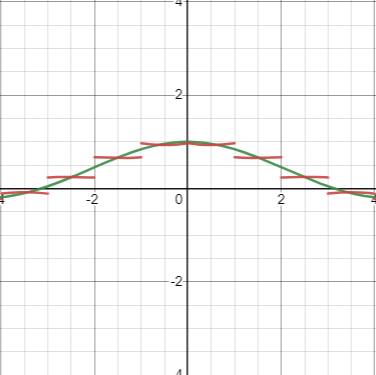}
    \caption{Approximation for $\frac{\sin x}x$}
    \label{figure 3}
\end{figure}
\FloatBarrier
\noindent Where $F(k)$ is some anti-derivative of $f(k)$. Using the fact that sums are solutions to the PDE $\frac{\partial^2}{\partial x\partial y}u=0$, we can differentiate with respect to $x$ on both sides to get $$\frac{\partial^2}{\partial x^2}\sum_{k=y}^xF(k)=\frac{\partial}{\partial x}\sum_{k=y}^xf(k)$$Replacing $x$ with $t$ then integrating with respect to $t$ once, then integrating again with respect to $t$ but with bounds being $y-1$ and $x$: $$\sum_{k=y}^xF(k)=\int_{y-1}^x\sum_{k=y}^tf(k)dt+c_1(y)(x-y+1)$$ Letting $x=y$, we get that $c_1(y)=F(y)-\int_{y-1}^y\sum_{k=y}^tf(k)dt$. Therefore $$\sum_{k=y}^xF(k)=\int_{y-1}^x\sum_{k=y}^tf(k)dt+\left(F(y)-\int_{y-1}^y\sum_{k=y}^tf(k)dt\right)(x-y+1)$$Denoted as $\star$. We can use this to find the sum of $k$. Letting $y=1$ for the sake of simplicity and $f(x)=1$, we get $$\sum_{k=1}^xk=\int_0^xtdt+\left(1-\int_0^1tdt\right)x=\frac{x^2+x}2$$The integration can be repeated, letting us compute:$$\sum_{k=1}^x\frac{k^2}2=\int_0^x\frac{t^2+t}2dt+\left(\frac12-\int_0^1\frac{t^2+t}2dt\right)x=\frac{x^3}6+\frac{x^2}4+\frac{x}{12}=\frac{x(x+1)(2x+1)}{12}$$ This integral formula can also be used to obtain the equivalence between the Digamma function and the harmonic numbers.$$\sum_{k=1}^x\frac{k^2}2=\int_0^x\frac{t^2+t}2dt+\left(\frac12-\int_0^1\frac{t^2+t}2dt\right)x=\frac{x^3}6+\frac{x^2}4+\frac{x}{12}=\frac{x(x+1)(2x+1)}{12}$$ This integral formula can also be used to obtain the equivalence between the Digamma function and the harmonic numbers. Letting $f(k)=\frac1k$ we get $$\sum_{k=1}^x\ln(k)=\ln(x!)=\int_0^xH_tdt-\gamma x$$Differentiating with respect to $x$ and adding by $\gamma$ we just get $\psi(x+1)+\gamma=H_x$. One more example is if we let $f(k)=\frac1{k^m}$:$$-(m-1)H_x^{(m-1)}=\int_0^xH_t^{(m)}dt-\zeta(m)x$$
There is still a limit to integrating the summand, as we can't sum functions like $f(k)=e^k$ and any function that if differentiated a given number of times, it would never reach a function that approaches a limit. We can, however, use Faulhaber's formula and the power series of such a function because $$\sum_{k=1}^xf(k)=\sum_{k=1}^x\sum_{i=0}^\infty c_ik^i=\sum_{i=0}^\infty c_i\sum_{k=1}^xk^i=\sum_{i=0}^\infty \frac{c_i}{i+1}\sum_{k=0}^i\binom{i+1}{k}B_kx^{i-k+1}$$Where $B_k$ are the Bernoulli numbers. And so we have another continuation for sums$$\sum_{k=y}^xf(k)=\sum_{i=0}^\infty \frac{c_i}{i+1}\sum_{k=0}^i\binom{i+1}{k}B_k(x^{i-k+1}-(y-1)^{i-k+1})$$If the power series was centered around $a\ne0$ then we have $$\sum_{k=y}^xf(k)$$$$=\sum_{j=0}^\infty c_j(-a)^j\left(\sum_{i=0}^j\frac{(-a)^{-i}}{i+1}\binom{j}{i}\sum_{k=0}^i\binom{i+1}{k}B_k(x^{i-k+1}-(y-1)^{i-k+1})\right)$$These formulae only work given that the infinite series converge. If there is a Taylor series for $f(x)$ then this would probably happen. They also aren't as easy to evaluate as Lemma 2.1 because of the binomial coefficients and Bernoulli numbers, but these new continuations don't require $f$ to converge to a constant.

Now again, starting with the PDE $$\left(\frac{\partial}{\partial x}+\frac{\partial}{\partial y}\right)\sum_{k=y}^xf(k)=\sum_{k=y}^xf'(k)$$And instead taking the derivative with respect to the lower bound, we get another formula for summing antiderivatives: $$\sum_{k=y}^xF(k)=\int_{x+1}^y\sum_{k=t}^xf(k)dt+\left(F(x)+\int_x^{x+1}\sum_{k=t}^xf(k)dt\right)(x-y+1)$$Which clearly implies that the LHS of $(\star)$ and this formula are both the same.

Notice that in the method of deriving these formulae, we let $x=y$ to be able to solve for $c_1(y)$ or $c_1(x)$, respectively. It's important to note that, because each one should depend on $y$ or $x$, we must replace $x$ with $y$ or $y$ with $x$ depending on the variable of the function. The reason why this is stated is because one could've derived $c_1(y)=F(x)-\int_{x-1}^x\sum_{k=y}^tf(k)dt$ without sensing a mistake.
\section{Discussion}
There is a restriction to Lemma 2.1 and its generalization, which needs $f(x)$ to approach a constant $L$ and be monotonic to have an analytic continuation. The formula for products is more restrictive as $L>0$. It can be proved that the infinite sum converges for specific $f$ but it would be better for the monotonocity requirement to be dropped to make the theorem more general. It is possible to replace this lemma with Theorem 6.1.1 in Summability Calculus as this isn't necessary, but the limit involved makes it difficult to manipulate its expression.

The two formulae given near the end of the "Summing Integrals" section are too complicated for practical use, and we need that $f$ has a power series that converges, which isn't always true. They also sort of defeat the purpose of this paper, but at least they can continuate almost any summand with a convergent power series. Hopefully, there will be future papers on these matters.

\printbibliography

@misc{mueller2007fractional,
      title={Fractional Sums and Euler-like Identities}, 
      author={Markus Mueller and Dierk Schleicher},
      year={2007},
      eprint={math/0502109},
      archivePrefix={arXiv},
      primaryClass={math.CA}
}

@misc{alabdulmohsin2012summability,
      title={Summability Calculus}, 
      author={Ibrahim M. Alabdulmohsin},
      year={2012},
      eprint={1209.5739},
      archivePrefix={arXiv},
      primaryClass={math.CA}
}
\end{document}